\documentclass[12pt]{article}
\usepackage{amsmath,amssymb,amsthm}
\usepackage{graphicx}
\usepackage{url}
\usepackage{cite}
\usepackage{xcolor}
\usepackage{mathtools}
\usepackage{epsfig}
\usepackage{eepic}
\usepackage{psfrag}
\usepackage{hyperref}
\advance \topmargin by -\headheight
\advance \topmargin by -\headsep
\evensidemargin \oddsidemargin
\def\Maketitle{\maketitle}

\makeatletter
\def\Appendix{\appendix
	\def\@seccntformat##1{Appendix~\csname the##1\endcsname.~~}}
\makeatother

\makeatletter
\@addtoreset{equation}{section}

\makeatother

\theoremstyle{plain}
\newtheorem{theorem}{Theorem}[section]
\newtheorem{lemma}[theorem]{Lemma}
\newtheorem{proposition}[theorem]{Proposition}

\newtheorem{remark}[theorem]{Remark}

\theoremstyle{definition}

\renewenvironment{proof}[1][Proof]{\noindent\textbf{#1. }}{\qed}

\def\XXint#1#2#3{{\setbox0=\hbox{$#1{#2#3}{\int}$}
		\vcenter{\hbox{$#2#3$}}\kern-.5\wd0}}

\begin{document}
\title {Painlev\'{e} IV, bi-confluent Heun equations and the Hankel determinant generated by a discontinuous semi-classical Laguerre weight}
\author{Mengkun Zhu$^{1}$\footnote{zmk@qlu.edu.cn}, Jianduo Yu$^{1}$\footnote{yujd0203@163.com} \\
\small $^{1}$ School of Mathematics and Statistics, Qilu University of Technology (Shandong Academy of Sciences),\\
\small Jinan, 250353, China\\
\small $^{2}$ College of Humanities and Education, Jiangmen Polytechnic, Jiangmen, 529000, China}
\date{}
\Maketitle
\begin{abstract}\vspace*{10pt}
We consider the discontinuous semi-classical Laguerre weight function with a jump
\begin{equation*}
w(x;t,s)={\rm e}^{-x^2+tx}(A+B\theta(x-s)),\quad x\in \mathbf{R}, t,s\geq0,\quad A\geq0,~A+B\geq0,
\end{equation*}
where $\theta(x)$ is 1 for $x > 0$ and 0 otherwise. Based on the ladder operator approach, we obtain some important difference and differential equations about the auxiliary quantities and the recurrence coefficients. By proper tranformation, It is shown that $R_{n}(t,s)$ is related to Painlev\'{e} IV equations and $r_{n}(t,s)$ satisfies the Chazy II equations. With the aid of Dyson's Coulomb fluid approach, we derive the asymptotic expansions for $\alpha_{n}$ and $\beta_{n}$ as $n\rightarrow\infty$. Furthermore, This enables us to obtain the lagre $n$ behavior of the orthogonal polynomials and derive that they satisfy the biconfluent Heun equation. We also consider the Hankel determinant $D_{n}(t,s)$ generated by the discountinuous semi-classical Laguerre weight. We find that the quantity $\sigma_{n}(t,s)$, allied to the logarithmic derivative of $D_{n}(t,s)$, satisfies the Jimbo-Miwa-Okamoto $\sigma$-form of Painlev\'{e} IV.

\end{abstract}



\section {Introduction}
The joint probability density function of a unitary ensemble in random matrix theory has the following form \cite{Mehta}
$$
P(x_{1},x_{2},\cdots,x_{n}):=\frac{1}{n !D_{n}[w_0]}\prod_{1 \leq i<j \leq n}\left(x_i-x_j\right)^2 \prod_{k=1}^n w_0\left(x_k\right) d x_k,
$$
where $w_0(x)$ is a weight function supported on an interval $[a, b]$, $\left\{x_j\right\}_{j=1}^n$ are the eigenvalues of $n \times n$ Hermitian matrices and $D_{n}[w_0]$ is the normalization constant
$$
D_n\left[w_0\right]:=\frac{1}{n !} \int_{[a, b]^n} \prod_{1 \leq i<j \leq n}\left(x_i-x_j\right)^2 \prod_{k=1}^n w_0\left(x_k\right) d x_k.
$$
The moments of $w_0(x)$ exist, namely,
$$
\mu_j= \int_{a}^{b} x w_0(x)d x, \quad j=0,1,2,\cdots
$$

In this paper, we study the discontinuous semi-classical Laguerre weight function
\begin{equation}\label{1.1}
w(x;t,s)=w_{0}(x)(A+B\theta(x-s)),
\end{equation}
where $x\in \mathbf{R}, t, s\geq0, A\geq0,~A+B\geq0$ and
\begin{equation*}
w_{0}(x)=\mathrm{e}^{-x^2+tx}.
\end{equation*}
More normally, we have
$$
D_n[w]:=\frac{1}{n !} \int_{\mathbf{R}^n} \prod_{1 \leq i<j \leq n}\left(x_i-x_j\right)^2 \prod_{k=1}^n w\left(x_k\right) d x_k,
$$
which is the Hankel determinant generated by the weight (\ref{1.1}). From \cite{Ismail}, it is known that $D_n[w]$ has another representations as follows:
\begin{align}\label{1.7}
D_{n}(t,s)=\prod_{j=0}^{n-1}h_{j}(t,s),
\end{align}
where $h_{j}(t,s)$ is the square of $L^2$ norm of the monic polynomials $P_{j}(x;t,s)$.
It is easy to see
\begin{align}\label{1.8}
\mathrm{ln}D_{n}(t,s)=\sum_{j=0}^{n-1}\mathrm{ln}h_{j}(t,s).
\end{align}
The inspiration of this paper originates from the work of \cite{CP2005}, in which they considered the Hankel determinant generated by the Hermite weight with one jump. Furthermore, the Hankel determinant for the Laguerre weight and Jacobi weight with a single jump were studied \cite{BC,CZ}. Min and Chen \cite{MC2019} investigated the Hankel determinants generated by a discontinuous Gaussian weight
$$
w(x; t_{1},t_{2}) =\mathrm{e}^{-x^2}(A+B_1 \theta (x-t_{1})+B_2 \theta (x-t_{2})).
$$
Different from \cite{CP2005}, \cite{MC2019} considered two general cases for the Gaussian weight: single jump($B_2=0$) and two jumps. By
using the ladder operator approach, they obtain difference and differential equations to describe the Hankel determinant, which include the Chazy II equation, continuous and discrete Painlev\'{e} IV for the single jump case, and showed that a quantity related to the Hankel determinant satisfies a two
variables generalization of the Jimbo-Miwa-Okamoto $\sigma$-form of the Painlev\'{e} IV with two jumps. Lyu and Chen \cite{LC2020} utilized the results about two jumps in \cite{MC2019} to derive the four auxiliary quantities allied with the orthogonal polynomials satisfy a coupled Painlev\'{e} IV system and the asymptotics for the recurrence coefficients which are connected with the solutions of the coupled Painlev\'{e} II system. Zhu, Wang and Chen \cite{WZC} studied a discontinuous linear statistic of the semi-classical Laguerre unitary ensembles
$$
w(x; t) = A\theta (x-t)\mathrm{e}^{-x^2+tx}.
$$
They obtained an auxiliary quantity satisfies a particular Painlev\'{e} IV equation and when $n$ gets large, the second order linear differential equation satisfied by
the orthogonal polynomials become a particular bi-confluent Heun equation.

Let the monic polynomials $P_{n}(x;t,s)$ of degree $n$ orthogonal with respect to the weight function $w(x, t,s)$, that is,
\begin{equation}\label{1.44}
\int^{\infty}_{0}P_{i}(x;t,s)P_{j}(x;t,s)\exp(-x^2+tx)(A+B\theta(x-s))dx=h_{i}(t,s)\delta_{ij},\quad i,j=0,1,2,\cdots,
\end{equation}
and $P_{n}(x;t,s)$ have the monomial expansion
\begin{equation}\label{1.55}
P_{n}(x;t,s)=x^n+\mathrm{p}(n,t,s)x^{n-1}+\cdots+P_{n}(0;t,s).
\end{equation}
From the orthogonality condition (\ref{1.44}), we have the three-term recurrence relation \cite{Szeg}
\begin{equation}\label{1.2}
xP_n(x;t,s)=P_{n+1}(x;t,s)+\alpha_nP_{n}(x;t,s)+\beta_nP_{n-1}(x;t,s),
\end{equation}
with the initial conditions
\begin{equation*}
P_0(x;t,s)=1,\quad \beta_0P_{-1}(x;t,s)=0.
\end{equation*}
Substituting (\ref{1.55}) into (\ref{1.2}) produces
\begin{align}\label{1.3}
\alpha_{n}(t,s)&=\mathrm{p}(n,t,s)-\mathrm{p}(n+1,t,s).
\end{align}
A telescopic sum of (\ref{1.3}) gives
\begin{align}\label{1.5}
\sum_{k=0}^{n-1}\alpha_{k}(t,s)&=-\mathrm{p}(n,t,s).
\end{align}
From (\ref{1.44}) and (\ref{1.2}), we have
\begin{align}\label{1.4}
\beta_{n}(t,s)&=\frac{h_{n}(t,s)}{h_{n-1}(t,s)}.
\end{align}
This paper is organized as follows: In section 2, using the ladder operators and Its compatibility conditions, we derive that the auxiliary quantity $R_{n}(t,s)$ satisfies the second order difference equation. In section 3, we consider the evolutions of the auxiliary quantities in $t$ and $s$. It is found that $R_{n}(t,s)$ and $r_{n}(t,s)$ satisfy the coupled Riccati equations, which composed of four equations and they also satisfy the second order partial differential equations, which can be transformed into Painlev\'{e} IV and Chazy II equations, respectively. In section 4, the difference equation and partial differential equation satisfy by the logarithmic derivative of Hankel determinant are derived, from which we find that they relate to discrete $\sigma$-form of the Painlev\'{e} IV equation and Jimbo-Miwa-Okamoto $\sigma$-form of the Painlev\'{e} IV equation, respectively. In section 5, using Dyson's Coulomb fluid approach, we study the asymptotic expansions of recurrence coefficients as $n\rightarrow\infty$. Furthermore, we find that the monic orthogonal polynomial $P_{n}(x; t,s)$ satisfies the biconfluent Heun equation as $n$ goes to infinity in section 5.

For convenience, we shall not display the $t$ and $s$ dependence in $P_{n}(x), w(x), h_{n}, \alpha_{n}$ and $\beta_{n}$ unless it is needed in the following discussions.

\section{Nonlinear difference equations and differential equations}\label{IM}
The ladder operators and Its compatibility conditions are very useful to study the recurrence coefficients of the orthogonal polynomials and the related Hankel determinants. See \cite{CF1,LCX,MC2019,MC2020} for references. From the works about discontinuous weights with one or two jumps, we have the following lemmas about lowering ladder operator for our monic orthogonal polynomial:

\begin{lemma}\label{1}
The monic orthogonal polynomials $P_{n}(x)$ with respect to the weight $w(x)$ satisfy the following recurrence relation:
\begin{gather*}
P'_{n}(x)=\beta_{n}A_{n}(x)P_{n-1}(x)-B_{n}(x)P_{n}(x),
\end{gather*}
where
\begin{equation}\label{2.1}
\begin{aligned}
A_{n}(x):&=\frac{BP^{2}_{n}(s,s,t)\mathrm{e}^{-s^2+ts}}{h_{n}(s,t)(x-s)}+\frac{1}{h_{n}(s,t)}\int^{+\infty}_{-\infty}\frac{v'_{0}(x)-v_{0}'(y)}{x-y}P^{2}_{n}(y)w(y)dy,\\
B_{n}(x):&=\frac{BP_{n}(s,s,t)P_{n-1}(s,s,t)\mathrm{e}^{-s^2+ts}}{h_{n-1}(s,t)(x-s)}+\frac{1}{h_{n-1}(s,t)}\int^{+\infty}_{-\infty}\frac{v_{0}'(x)-v_{0}'(y)}{x-y}P_{n}(y)P_{n-1}(y)w(y)dy,
\end{aligned}
\end{equation}
with $v_{0}(x)=x^2-tx, P_{n}(s,s,t):=P_{n}(x,s,t)|_{x=s}$.

\end{lemma}

Due to the discontinuity of $\theta(x-s)$ at $s$, $A_{n}(x)$ and $B_{n}(x)$ have apparent poles at $s$ with residues $R_{n}$ and $r_{n}$, respectively.
\begin{lemma}\label{2}The monic orthogonal polynomials $P_{n}(x)$ satisfy the following second order differential equation
\begin{equation}\label{2.555}
\begin{aligned}
&P''_{n}(x)-\bigg(\frac{A'_{n}(x)}{A_{n}(x)}+v'_{0}(x)\bigg)P_{n}'(x)+\bigg[B'_{n}(x)-\frac{A'_{n}(x)B_{n}(x)}{A_{n}(z)}+\sum_{j=0}^{n-1}A_{j}(x)\bigg]P_{n}(x)=0.
\end{aligned}
\end{equation}
\end{lemma}

The functions $A_{n}(x)$ and $B_{n}(x)$ satisfy the following compatibility conditions:

\begin{subequations}
\begin{align}
&B_{n+1}(x)+B_{n}(x)=(x-\alpha_{n})A_{n}(x)-v_{0}'(x),\label{2.4a}\\
&1+(x-\alpha_{n})(B_{n+1}(x)-B_{n}(x))=\beta_{n+1}A_{n+1}(x)-\beta_{n}A_{n-1}(x),\label{2.4b}\\
&B^{2}_{n}(x)+v_{0}'(x)B_{n}(x)+\sum ^{n-1}_{j=0}A_{j}(x)=\beta_{n}A_{n}(x)A_{n-1}(x).\label{2.4c}
\end{align}
\end{subequations}

From these compatibility conditions, we can determine the important identity equations immediately about the recurrence coefficients $\alpha_{n}$, $\beta_{n}$ and other auxiliary quantities.

For our problem, we have
\begin{align*}
 v_{0}'(x)=2x-t,
\end{align*}and
\begin{align*}
\frac{v_{0}'(x)-v_{0}'(y)}{x-y}=2.
\end{align*}
Substituting the above equation into the definition of $A_{n}(x)$ and $B_{n}(x)$ in \eqref{2.1}, we have
\begin{subequations}
\begin{align}
A_{n}(x)&=\frac{R_{n}}{x-s}+2,\label{2.2a}\\
B_{n}(x)&=\frac{r_{n}}{x-s},\label{2.2b}
\end{align}
\end{subequations}
where
\begin{equation*}
\begin{aligned}
R_{n}:=\frac{BP^{2}_{n}(s,s,t)\mathrm{e}^{-s^2+ts}}{h_{n}(s,t)},\quad
r_{n}:=\frac{BP_{n}(s,s,t)P_{n-1}(s,s,t)\mathrm{e}^{-s^2+ts}}{h_{n-1}(s,t)}.
\end{aligned}
\end{equation*}

Plugging \eqref{2.2a} and \eqref{2.2b} into \eqref{2.4a}, we obtain
\begin{subequations}
\begin{align}
r_{n+1}+r_{n}=(s-\alpha_{n})R_{n},\label{2.5a}\\
R_{n}+t=2\alpha_{n}.\label{2.5b}
\end{align}
\end{subequations}
\eqref{2.4b} gives the following two equations:
\begin{subequations}
\begin{align}
r_{n+1}-r_{n}=2\beta_{n+1}-2\beta_{n}-1,\label{2.44a}\\
(\alpha_{n}-s)(r_{n+1}-r_{n})=\beta_{n}R_{n-1}-\beta_{n+1}R_{n+1}.\label{2.44b}
\end{align}
\end{subequations}
From \eqref{2.4c}, we get three identities:
\begin{subequations}
\begin{align}
&r^2_{n}=\beta_{n}R_{n}R_{n-1},\label{2.6a}\\
&\sum ^{n-1}_{j=0}R_{j}+(2s-t)r_{n}=2\beta_{n}(R_{n}+R_{n-1}),\label{2.6b}\\
&r_{n}+n=2\beta_{n}.\label{2.6c}
\end{align}
\end{subequations}
Based on these results, we can obtain some second order nonlinear difference equations satisfied by auxiliary quantities as following:
\begin{theorem}
The auxiliary quantity $R_{n}$ satisfies the following nonlinear difference equation:
\begin{equation}\label{2.7}
\begin{aligned}
[2R_{n}(2s-t-R_{n})^2-2(R_{n+1}+R_{n-1})(2sR_{n}-tR_{n}-R^2_{n}+2n)+R_{n}R_{n-1}R_{n+1}-4R_{n+1}]^2\\
=R_{n+1}R_{n-1}(R_{n}R_{n-1}+8n)(R_{n}R_{n+1}+8n+8),
\end{aligned}
\end{equation}
with the initial conditions
\begin{equation*}
\begin{aligned}
R_{0}=\frac{B\mathrm{e}^{-s^2+ts}}{h_{0}}, \quad R_{1}=\frac{BP^2_{1}(s;s,t)\mathrm{e}^{-s^2+ts}}{h_{0}}.
\end{aligned}
\end{equation*}
\end{theorem}

\begin{proof}
Replacing $\alpha_{n}$ in \eqref{2.5a} by \eqref{2.5b}, we have
\begin{equation}\label{2.9}
\begin{aligned}
2(r_{n}+r_{n+1})=(2s-t-R_{n})R_{n}.
\end{aligned}
\end{equation}
Eliminating $\beta_{n}$ from \eqref{2.6a} by using \eqref{2.6c} yields
\begin{equation*}
\begin{aligned}
2r^2_{n}=(r_{n}+n)R_{n}R_{n-1}.
\end{aligned}
\end{equation*}
Solving for $r_{n}$ from the above equation produces
\begin{equation}\label{2.10}
\begin{aligned}
r_{n}=\frac{1}{4}[R_{n}R_{n-1}\pm\sqrt{R_{n}R_{n-1}(R_{n}R_{n-1}+8n)}].
\end{aligned}
\end{equation}
Inserting either solution into \eqref{2.9}, we obtain \eqref{2.7} after removing the square root.

\end{proof}

\section{$s$ evolution and $t$ evolution}\label{IM}
 Because of the quantities, such as the recurrence coefficients $\alpha_{n}, \beta_{n}$ and the auxiliary quantities $R_{n}, r_{n}$, depend on $s$ and $t$, we devote to studying the evolution of all the quantities in $s$ and $t$.
We begin with considering the differentiation of
\begin{equation*}
\int^{\infty}_{0}P^{2}_{n}(x)\exp(-x^2+tx)(A+B\theta(x-s))dx=h_{n},
\end{equation*}
over $s$ and $t$, and we obtain
\begin{equation}\label{3.1}
\begin{aligned}
\partial_{s}\mathrm{ln}h_{n}=-R_{n},\\
\partial_{t}\mathrm{ln}h_{n}=\alpha_{n},
\end{aligned}
\end{equation}
where we used the recurrence relation \eqref{1.2} to obtain the second equation in \eqref{3.1}.

Using \eqref{1.4} and \eqref{3.1}, we have
\begin{equation*}
\begin{aligned}
\partial_{s}\mathrm{ln}\beta_{n}=\partial_{s}\mathrm{ln}h_{n}-\partial_{s}\mathrm{ln}h_{n-1}=R_{n-1}-R_{n},\\
\partial_{t}\mathrm{ln}\beta_{n}=\partial_{t}\mathrm{ln}h_{n}-\partial_{t}\mathrm{ln}h_{n}=\alpha_{n}-\alpha_{n-1}.
\end{aligned}
\end{equation*}

Computing the above equations yield
\begin{equation}\label{3.2}
\begin{aligned}
\partial_{s}\beta_{n}=\beta_{n}(R_{n-1}-R_{n}),
\end{aligned}
\end{equation}
\begin{equation}\label{3.3}
\begin{aligned}
\partial_{t}\beta_{n}=\beta_{n}(\alpha_{n}-\alpha_{n-1}).
\end{aligned}
\end{equation}

Taking a derivative with respect to $s$ in the equation
\begin{equation}\label{3.4}
\int^{\infty}_{0}P_{n}(x)P_{n-1}(x)\exp(-x^2+tx)(A+B\theta(x-s))dx=0,
\end{equation}
we have
\begin{equation}\label{3.5}
\begin{aligned}
\partial_{s}\mathrm{p}(n,t)=\frac{\mathrm{exp}\left(-s^2+ts\right)P_{n}(s)P_{n-1}(s)}{h_{n-1}}=r_{n}(s,t).
\end{aligned}
\end{equation}
Taking a derivative with respect to $t$ in \eqref{3.4} and using \eqref{1.2} produce
\begin{equation}\label{3.6}
\begin{aligned}
\partial_{t}\mathrm{p}(n,t)=-\beta_{n}.
\end{aligned}
\end{equation}

Now we have the analogs of Toda equations on $\alpha_{n}$ and $\beta_{n}$.
\begin{proposition}
\begin{equation}\label{3.7}
\begin{aligned}
\partial_{s}\beta_{n}=2\beta_{n}(\alpha_{n-1}-\alpha_{n}),
\end{aligned}
\end{equation}
\begin{equation}\label{3.8}
\begin{aligned}
\partial_{t}\beta_{n}=\beta_{n}(\alpha_{n}-\alpha_{n-1}),
\end{aligned}
\end{equation}
\begin{equation}\label{3.9}
\begin{aligned}
\partial_{s}\alpha_{n}=2\beta_{n}-2\beta_{n+1}+1,
\end{aligned}
\end{equation}
\begin{equation}\label{3.10}
\begin{aligned}
\partial_{t}\alpha_{n}=\beta_{n+1}-\beta_{n}.
\end{aligned}
\end{equation}
\end{proposition}

\begin{proof}
Equation \eqref{3.7} comes from \eqref{2.5b} and \eqref{3.2}. { \eqref{3.8} is \eqref{3.3}}. Taking a derivative with respect to $s$ and $t$ in \eqref{1.3}, we find
\begin{equation}\label{3.111}
\begin{aligned}
\partial_{s}\alpha_{n}=\partial_{s}\mathrm{p}(n)-\partial_{s}\mathrm{p}(n+1),
\end{aligned}
\end{equation}
\begin{equation}\label{3.112}
\begin{aligned}
\partial_{t}\alpha_{n}=\partial_{t}\mathrm{p}(n)-\partial_{t}\mathrm{p}(n+1).
\end{aligned}
\end{equation}
Inserting \eqref{3.5} into \eqref{3.111} and using \eqref{2.6c} to eliminate $r_{n}$, we obtain \eqref{3.9}. The combination of \eqref{3.6} and \eqref{3.112} results in \eqref{3.10}.
The following lemma is very crucial to obtain the Painlev\'{e} equations by derivation of the second-order differential equations satisfied by $R_{n}$ and $r_{n}$.
\end{proof}
\begin{lemma}
The quantities $R_{n}$ and $r_{n}$ satisfy the analogs Riccati equations:
\begin{equation}\label{3.11}
\begin{aligned}
\partial_{s}r_{n}=\frac{2r^{2}_{n}}{R_{n}}-R_{n}(t)(r_{n}+n),
\end{aligned}
\end{equation}
\begin{equation}\label{3.12}
\begin{aligned}
\partial_{t}r_{n}=-\frac{r^2_{n}}{R_{n}}+\frac{1}{2}R_{n}(r_{n}+n),
\end{aligned}
\end{equation}
\begin{equation}\label{3.13}
\begin{aligned}
\partial_{s}R_{n}=4r_{n}-R_{n}(2s-t-R_{n}),
\end{aligned}
\end{equation}
\begin{equation}\label{3.14}
\begin{aligned}
\partial_{t}R_{n}=\frac{1}{2}R_{n}(2s-t-R_{n})-2r_{n}.
\end{aligned}
\end{equation}
\end{lemma}

\begin{proof}
Getting rid of $\beta_{n}R_{n-1}$ in \eqref{3.2} by using \eqref{2.6a}, we have
\begin{equation}\label{3.15}
\begin{aligned}
\partial_{s}\beta_{n}=\frac{r_{n}^2}{R_{n}}-\beta_{n}R_{n}.
\end{aligned}
\end{equation}
Inserting \eqref{2.6c} into this identity, we arrive at \eqref{3.11}.

Plugging \eqref{2.5b} into \eqref{3.3} to remove $\alpha_{n}$, we have
\begin{equation}\label{3.16}
\begin{aligned}
\partial_{t}\beta_{n}=\frac{1}{2}\beta_{n}(R_{n}-R_{n-1}).
\end{aligned}
\end{equation}
Using \eqref{2.6a} and \eqref{2.6c}, we obtain \eqref{3.12}.

Substituting \eqref{1.3} into \eqref{2.5b}, we find
\begin{equation}\label{3.17}
\begin{aligned}
R_{n}+t=2[\mathrm{p}(n)-\mathrm{p}(n+1)].
\end{aligned}
\end{equation}
Taking a derivative on both sides of \eqref{3.17} with respect to $s$, we have
\begin{equation}\label{3.18}
\begin{aligned}
\partial_{s}R_{n}&=2[\partial_{s}\mathrm{p}(n)-\partial_{s}\mathrm{p}(n+1)]=2(r_{n}-r_{n+1})=2[2r_{n}-R_{n}(s-\alpha_{n})] \\
&=4r_{n}-2R_{n}(s-\frac{1}{2}(t+R_{n}))=4r_{n}-R_{n}(2s-t-R_{n}),
\end{aligned}
\end{equation}
where we used \eqref{3.6}, \eqref{2.5a} and \eqref{2.5b}.

Similarly, Taking a derivative on both sides of \eqref{3.17} with respect to $t$ and using \eqref{3.7} and \eqref{2.6c}, we have
\begin{equation}\label{3.19}
\begin{aligned}
\partial_{t}R_{n}+1&=2(\partial_{t}\mathrm{p}(n)-\partial_{t}\mathrm{p}(n+1))\\
&=2\beta_{n+1}-2\beta_{n}=r_{n+1}-r_{n}+1.
\end{aligned}
\end{equation}
Substituting \eqref{2.5a} into \eqref{3.19} to eliminate $r_{n+1}$ and using \eqref{2.5b} to remove $\alpha_{n}$, we have the desired identity \eqref{3.14}.

\end{proof}

\begin{theorem}
$R_{n}$ and $r_{n}$ satisfy the following second order partial differential equations:
\begin{equation}\label{3.20}
\begin{aligned}
\partial^{2}_{ss}R_{n}=\frac{(\partial_{s}R_{n})^2}{2R_{n}}+\frac{3}{2}R^3_{n}-(4s-2t)R^2_{n}+(2s^2-4n-2-2st+\frac{1}{2}t^2)R_{n},
\end{aligned}
\end{equation}
\begin{equation}\label{3.21}
\begin{aligned}
\partial^{2}_{tt}R_{n}=\frac{(\partial_{t}R_{n})^2}{2R_{n}}+\frac{3}{8}R^3_{n}-(s-\frac{1}{2}t)R^2_{n}+(\frac{1}{2}s^2-n-\frac{1}{2}-\frac{1}{2}st+\frac{1}{8}t^2)R_{n}.
\end{aligned}
\end{equation}
\begin{equation}\label{3.22}
\begin{aligned}
[\partial^{2}_{ss}r_{n}+12r^2_{n}+8nr_{n}]^2=(2s-t)^2[(\partial_{s}r_{n})^2+8r^3_{n}+8nr^2_{n}],
\end{aligned}
\end{equation}
\begin{equation}\label{3.23}
\begin{aligned}
[2\partial^{2}_{tt}r_{n}+6r^2_{n}+4nr_{n}]^2=(2s-t)^2[(\partial_{t}r_{n})^2+2r^3_{n}+2nr^2_{n}],
\end{aligned}
\end{equation}

Moreover, letting ~$S:=\dfrac{1}{2}t-s$, $Y(S):=R_{n}(s,t)$, then $Y(S)$ satisfies the Painlev\'{e} IV equation \cite{Gromak}
\begin{equation}\label{3.24}
\begin{aligned}
\partial^{2}_{SS}Y=\frac{(\partial_{S}Y)^2}{2Y}+\frac{3}{2}Y^3+4SY^2+2(S^2-\gamma_{1})Y+\frac{\eta_{1}}{Y},
\end{aligned}
\end{equation}
with $\gamma_{1}=2n+1$, $\eta_{1}=0$. Letting $T:=\dfrac{1}{2}t-s$, $Y(T):=R_{n}(s,t)$, then $Y(T)$ satisfies the Painlev\'{e} IV equation \cite{Gromak}
\begin{equation}\label{3.25}
\begin{aligned}
\partial^{2}_{TT}Y=\frac{(\partial_{T}Y)^2}{2Y}+\frac{3}{2}Y^3+4TY^2+2(T^2-\gamma_{1})Y+\frac{\eta_{1}}{Y}.
\end{aligned}
\end{equation}
Letting $\widetilde{S}:=s-\dfrac{1}{2}t$, $V(\widetilde{S}):=-2r_{n}(s,t)-\dfrac{2n}{3}$, then $V(\widetilde{S})$ satisfies the Chazy II system \cite{Cosgrove},
\begin{equation}\label{CS}
\begin{aligned}
[\partial^2_{\widetilde{S}\widetilde{S}}V-6V^2-\gamma_{2}]^2=4\widetilde{S}^2[(\partial_{\widetilde{S}}V)^2-4V^3-2\gamma_{2}V-\eta_{2}],
\end{aligned}
\end{equation}
with $\gamma_{2}=-\dfrac{8n^2}{3}$, $\eta_{2}=-\dfrac{64n^3}{27}$. Letting $\widetilde{T}:=s-\dfrac{1}{2}t$, $V(\widetilde{T}):=-2r_{n}(s,t)-\dfrac{2n}{3}$, then $V(\widetilde{T})$ satisfies the Chazy II system \cite{Cosgrove}
\begin{equation}\label{CT}
\begin{aligned}
[\partial^2_{\widetilde{T}\widetilde{T}}V-6V^2-\gamma_{2}]^2=4\widetilde{T}^2[(\partial_{\widetilde{T}}V)^2-4V^3-2\gamma_{2}V-\eta_{2}].
\end{aligned}
\end{equation}

\end{theorem}

\begin{proof}
Solving $r_{n}$ from \eqref{3.13}, we find
\begin{equation}\label{3.26}
\begin{aligned}
r_{n}=\frac{1}{4}[\partial_{s}R_{n}+(2s-t-R_n)R_n].
\end{aligned}
\end{equation}
Inserting \eqref{3.26} into \eqref{3.11}, the differential equation \eqref{3.20} for $R_{n}$ is obtained.

Expressing $r_{n}$ in terms of $R_{n}$ and $\partial_{t}R_{n}$ by \eqref{3.14} gives us
\begin{equation}\label{3.27}
\begin{aligned}
r_{n}=\frac{1}{4}R_{n}(2s-t-R_n)-\frac{1}{2}\partial_{t}R_{n}.
\end{aligned}
\end{equation}
Plugging \eqref{3.27} into \eqref{3.12}, we are led to \eqref{3.21}.

Letting $s=\frac{1}{2}t-S$, $R(s,t)=Y(S)$ and $t=2T+2s$, $R(s,t)=Y(T)$, substituting them into \eqref{3.20} and \eqref{3.21}, respectively, we obtain $Y(S)$ and $Y(T)$ satisfy the particular Painlev\'{e} IV.

Now, we derive the differential equations for $r_{n}$. Viewing \eqref{3.11} as the quadratic equation on $R_{n}$, we find
\begin{equation}\label{3.28}
\begin{aligned}
R_{n}=\frac{-\partial_{s}r_n\pm\sqrt{(\partial_{s}r_n)^2+8r_{n}^2(n+r_{n})}}{2(n+r_{n})}.
\end{aligned}
\end{equation}
Then, inserting one of the solutions $R_{n}=\frac{-\partial_{s}r_n+\sqrt{(\partial_{s}r_n)^2+8r_{n}^2(n+r_{n})}}{2(n+r_{n})}$ into \eqref{3.13}, we have
\begin{equation}\label{3.29}
\begin{aligned}
\frac{[\partial_{s}r_n-\sqrt{(\partial_{s}r_n)^2+8r_{n}^2(n+r_{n})}][\partial^2_{ss}r_n+12r_{n}^2+8nr_{n}-(2s-t)\sqrt{(\partial_{s}r_n)^2+8r_{n}^2(n+r_{n})}]}{2(n+r_{n})\sqrt{(\partial_{s}r_n)^2+8r_{n}^2(n+r_{n})}}=0.
\end{aligned}
\end{equation}
It follows a second order differential equation for $r_{n}$,
\begin{equation}\label{3.30}
\begin{aligned}
\partial^2_{ss}r_n+12r_{n}^2+8nr_{n}-(2s-t)\sqrt{(\partial_{s}r_n)^2+8r_{n}^2(n+r_{n})}=0.
\end{aligned}
\end{equation}
Clearing the square root, we obtain \eqref{3.22}.

\eqref{3.12} can be seen as the quadratic equation of $R_{n}$,
\begin{equation}\label{3.31}
\begin{aligned}
R_{n}=\frac{\partial_{t}r_n\pm\sqrt{(\partial_{t}r_n)^2+2r_{n}^2(n+r_{n})}}{n+r_{n}}.
\end{aligned}
\end{equation}
Substituting the solution $R_{n}=\dfrac{\partial_{t}r_n+\sqrt{(\partial_{t}r_n)^2+2r_{n}^2(n+r_{n})}}{n+r_{n}}$ into \eqref{3.14}, we have
\begin{equation}\label{3.32}
\begin{aligned}
\frac{[\partial_{t}r_n+\sqrt{(2\partial_{t}r_n)^2+2r_{n}^2(n+r_{n})}][2\partial^2_{tt}r_n+6r_{n}^2+4nr_{n}-(2s-t)\sqrt{(\partial_{t}r_n)^2+2r_{n}^2(n+r_{n})}]}{2(n+r_{n})\sqrt{(\partial_{t}r_n)^2+2r_{n}^2(n+r_{n})}}=0.
\end{aligned}
\end{equation}
Choosing the second order differential equation for $r_{n}$ and removing the square root, we obtain \eqref{3.23}.

Letting $s=\widetilde{S}+\dfrac{1}{2}t$, $r_{n}(s,t)=-\dfrac{1}{2}V(\widetilde{S})-\dfrac{n}{3}$, then $V(\widetilde{S})$ satisfies the Chazy II system \eqref{CS}. Similarly, letting $t=2s-2\widetilde{T}$, $r_{n}(s,t)=-\dfrac{1}{2}V(\widetilde{T})-\dfrac{n}{3}$, then $V(\widetilde{T})$ satisfies the Chazy II system \eqref{CT}.
\end{proof}

\begin{remark}
One finds that substituting the another solution in \eqref{3.28} into \eqref{3.13} also can obtain the desired result. Similarly, plugging the another solution in \eqref{3.31} into \eqref{3.14} can obtain the same result.

\end{remark}

\section{Logarithmic derivative of the Hankel determinant and $\sigma$-form of Panilev\'{e} IV}\label{IM}

We define a quantity related to the logarithmic derivative of the Hankel determinant $D_{n}$,
\begin{equation}\label{4.1}
\begin{aligned}
\sigma_{n}=(\partial_{t}+\partial_{s})\mathrm{ln}D_{n}.
\end{aligned}
\end{equation}
It is easy to find from \eqref{1.8} and \eqref{3.1} that
\begin{equation}\label{4.2}
\begin{aligned}
\sigma_{n}=-\sum\limits_{j=0}^{n-1}R_{j}+\sum\limits_{j=0}^{n-1}\alpha_{j}.
\end{aligned}
\end{equation}
Replacing $n$ by $j$ in \eqref{2.5b} and summing it from $j=0$ to $n-1$, we have
\begin{equation}\label{4.3}
\begin{aligned}
\sum\limits_{j=0}^{n-1}R_{j}=2\sum\limits_{j=0}^{n-1}\alpha_{j}-nt.
\end{aligned}
\end{equation}
Substituting \eqref{4.3} into \eqref{4.2} produces
\begin{equation}\label{4.4}
\begin{aligned}
2\sigma_{n}=-\sum\limits_{j=0}^{n-1}R_{j}+nt.
\end{aligned}
\end{equation}
Then we have
\begin{equation}\label{4.5}
\begin{aligned}
R_{n}=2\sigma_{n}-2\sigma_{n+1}+t.
\end{aligned}
\end{equation}
On the other hand, inserting \eqref{4.3} into \eqref{4.2} to eliminate $\sum\limits_{j=0}^{n-1}R_{j}$, we have
\begin{equation}\label{4.6}
\begin{aligned}
\sigma_{n}=-\sum\limits_{j=0}^{n-1}\alpha_{j}+nt.
\end{aligned}
\end{equation}
In view of \eqref{1.5}, we have
\begin{equation}\label{4.7}
\begin{aligned}
\sigma_{n}=nt+\mathrm{p}.
\end{aligned}
\end{equation}

We establish the following theorem which can be seen as a two variable version of Toda-type equation.
\begin{theorem}
$D_{n}$ satisfies the following second order partial differential equation:
\begin{equation*}
\begin{aligned}
(\partial^2_{tt}+2\partial^2_{st}+\partial^2_{ss})\mathrm{ln}D_{n}=\frac{D_{n+1}D_{n-1}}{(D_{n})^2}.
\end{aligned}
\end{equation*}
\end{theorem}
\begin{proof}
Applying $\partial_{t}+\partial_{s}$ again to \eqref{4.7}, in view of \eqref{3.5} and \eqref{3.6},we find
 \begin{equation}\label{4.8}
\begin{aligned}
(\partial^2_{tt}+2\partial^2_{st}+\partial^2_{ss})\mathrm{ln}D_{n}=n+r_{n}-\beta_{n}.
\end{aligned}
\end{equation}
Getting rid of $r_{n}$ by using \eqref{2.6c}, \eqref{4.8} becomes
\begin{equation}\label{4.9}
\begin{aligned}
(\partial^2_{tt}+2\partial^2_{st}+\partial^2_{ss})\mathrm{ln}D_{n}=\beta_{n}.
\end{aligned}
\end{equation}
From \eqref{1.7}, we have
\begin{equation}\label{4.10}
\begin{aligned}
h_{n}=\frac{D_{n+1}}{D_{n}}.
\end{aligned}
\end{equation}
With the help of \eqref{1.4} and \eqref{4.10}
\begin{equation}\label{4.11}
\begin{aligned}
\beta_{n}=\frac{D_{n+1}D_{n-1}}{(D_{n})^2}.
\end{aligned}
\end{equation}
Plugging \eqref{4.11} into \eqref{4.9}, we completed the proof.

\end{proof}
\begin{theorem}
$\sigma_{n}$ satisfies the following the second-order difference equation and the second order partial differential equations:
\begin{equation}\label{4.12}
\begin{aligned}
&[nt+2\sigma_{n}+2n(\sigma_{n-1}-\sigma_{n+1})]^2\\
&~~~~~~~~~~~~~~=(2\sigma_{n-1}-2\sigma_{n+1}+3t-2s)(nt-ns-\sigma_{n})(2\sigma_{n}-2\sigma_{n+1}+t)(2\sigma_{n-1}-2\sigma_{n}+t),
\end{aligned}
\end{equation}
\begin{equation}\label{4.13}
\begin{aligned}
(\partial^2_{ss}\sigma_{n})^2=[(2s-t)\partial_{s}\sigma_{n}+nt-2\sigma_{n}]^2-8(\partial_{s}\sigma_{n}+n)(\partial_{s}\sigma_{n})^2,
\end{aligned}
\end{equation}
\begin{equation}\label{4.14}
\begin{aligned}
16(\partial^2_{tt}\sigma_{n})^2=[(2s-t)(n-2\partial_{t}\sigma_{n})+nt-2\sigma_{n}]^2-16(n-\partial_{t}\sigma_{n})(n-2\partial_{t}\sigma_{n})^2.
\end{aligned}
\end{equation}
Letting $\widetilde{S}:=s-\dfrac{1}{2}t, \widetilde{\sigma}_{n}:=2\sigma_{n}-n(2s-2\widetilde{S})$, \eqref{4.12} becomes a discrete $\sigma$-form of the Painlev\'{e} IV equation
\begin{equation}\label{4.15}
\begin{aligned}
2[\widetilde{\sigma}_{n}+n(\widetilde{\sigma}_{n-1}-\widetilde{\sigma}_{n+1})]^2
=(\widetilde{\sigma}_{n+1}-\widetilde{\sigma}_{n-1}+2\widetilde{S})(2n\widetilde{S}+\widetilde{\sigma}_{n})(\widetilde{\sigma}_{n}-\widetilde{\sigma}_{n+1})(\widetilde{\sigma}_{n-1}-\widetilde{\sigma}_{n}).
\end{aligned}
\end{equation}
and \eqref{4.13} becomes a Jimbo-Miwa-Okamoto $\sigma$-form of the Painlev\'{e} IV equation \cite{Jimbo}
\begin{equation}\label{4.16}
\begin{aligned}
(\partial^2_{\widetilde{S}\widetilde{S}}\widetilde{\sigma}_{n})^2=4(\widetilde{S}\partial_{\widetilde{S}}\widetilde{\sigma}_{n}-\widetilde{\sigma}_{n})^2
-4(\partial_{\widetilde{S}}\widetilde{\sigma}_{n})^2(\partial_{\widetilde{S}}\widetilde{\sigma}_{n}+2n).
\end{aligned}
\end{equation}
Letting $\widetilde{T}:=s-\dfrac{1}{2}t, \widehat{\sigma}_{n}:=2\sigma_{n}-n(2s-2\widetilde{T})$, \eqref{4.14} becomes a Jimbo-Miwa-Okamoto $\sigma$-form of the Painlev\'{e} IV equation \cite{Jimbo}
\begin{equation}\label{4.17}
\begin{aligned}
(\partial^2_{\widetilde{T}\widetilde{T}}\widehat{\sigma}_{n})^2=4(\widetilde{T}\partial_{T}\widehat{\sigma}_{n}-\widehat{\sigma}_{n})^2
-4(\partial_{\widetilde{T}}\widehat{\sigma}_{n})^2(\partial_{\widetilde{T}}\widehat{\sigma}_{n}+2n).
\end{aligned}
\end{equation}
\end{theorem}

\begin{proof}
Using \eqref{2.6c}, \eqref{4.4} and \eqref{4.5}, \eqref{2.6b} becomes
\begin{equation}\label{4.18}
\begin{aligned}
nt-2\sigma_{n}=(r_{n}+n)(2\sigma_{n-1}-2\sigma_{n+1}+2t)+(t-2s)r_{n}.
\end{aligned}
\end{equation}
Solving for $r_{n}$ from the above equation, we have
\begin{equation}\label{4.19}
\begin{aligned}
r_{n}=\frac{-nt-2\sigma_{n}-2n(\sigma_{n-1}-\sigma_{n+1})}{2\sigma_{n-1}-2\sigma_{n+1}+3t-2s}.
\end{aligned}
\end{equation}
With the help of \eqref{2.6c} and \eqref{4.5}, it follows from \eqref{2.6a} that
\begin{equation}\label{4.20}
\begin{aligned}
2r^2_{n}=(r_{n}+n)(2\sigma_{n}-2\sigma_{n+1}+t)(2\sigma_{n-1}-2\sigma_{n}+t).
\end{aligned}
\end{equation}
Inserting \eqref{4.19} into \eqref{4.20}, we obtain \eqref{4.12}.

Using \eqref{2.6a}, \eqref{2.6c} and \eqref{4.4}, \eqref{2.6b} becomes
\begin{equation}\label{4.21}
\begin{aligned}
\frac{2r^2_{n}}{R_{n}}+(r_{n}+n)R_{n}+(t-2s)r_{n}=nt-2\sigma_{n}.
\end{aligned}
\end{equation}
The minus and sum of \eqref{3.11} and \eqref{4.21} gives
\begin{equation}\label{4.22}
\begin{aligned}
\frac{4r^2_{n}}{R_{n}}=nt-2\sigma_{n}-(t-2s)r_{n}+\partial_{s}r_{n},
\end{aligned}
\end{equation}
and
\begin{equation}\label{4.23}
\begin{aligned}
2(r_{n}+n)R_{n}=nt-2\sigma_{n}-(t-2s)r_{n}-\partial_{s}r_{n}.
\end{aligned}
\end{equation}
The product of \eqref{4.22} and \eqref{4.23} leads to
\begin{equation}\label{4.24}
\begin{aligned}
8r^2_{n}(r_{n}+n)=[nt-2\sigma_{n}-(t-2s)r_{n}]^2-(\partial_{s}r_{n})^2.
\end{aligned}
\end{equation}

Differential \eqref{4.7} over $s$ to obtain
\begin{equation}\label{4.25}
\begin{aligned}
\partial_{s}\sigma_{n}=r_{n}.
\end{aligned}
\end{equation}
Substituting \eqref{4.25} into \eqref{4.24}, we come to \eqref{4.13}.

To proceed,
The minus and sum of \eqref{3.12} and \eqref{4.21} gives
\begin{equation}\label{4.26}
\begin{aligned}
\frac{4r^2_{n}}{R_{n}}=nt-2\sigma_{n}-(t-2s)r_{n}-2\partial_{t}r_{n},
\end{aligned}
\end{equation}
and
\begin{equation}\label{4.27}
\begin{aligned}
2(r_{n}+n)R_{n}=nt-2\sigma_{n}-(t-2s)r_{n}+2\partial_{t}r_{n}.
\end{aligned}
\end{equation}
The product of \eqref{4.26} and \eqref{4.27} leads to
\begin{equation}\label{4.28}
\begin{aligned}
8r^2_{n}(r_{n}+n)=[nt-2\sigma_{n}-(t-2s)r_{n}]^2-4(\partial_{t}r_{n})^2.
\end{aligned}
\end{equation}
Since $\partial_{t}\mathrm{p}(n)=-\beta_{n}$, we find from \eqref{4.7} that
\begin{equation}\label{4.29}
\begin{aligned}
\partial_{t}\sigma_{n}=n-\beta_{n}=\frac{1}{2}(n-r_{n}).
\end{aligned}
\end{equation}
Solving $r_{n}$ from \eqref{4.29} gives
\begin{equation}\label{4.30}
\begin{aligned}
r_{n}=n-2\partial_{t}\sigma_{n}.
\end{aligned}
\end{equation}
Inserting \eqref{4.30} into \eqref{4.28}, we obtain \eqref{4.14}.

In addition, letting
\begin{equation}\label{4.31}
	\begin{aligned}
		s=\frac{1}{2}t+\widetilde{S}, \quad \sigma_{n}(s,t)=\dfrac{1}{2}\widetilde{\sigma}_{n}+n(s-\widetilde{S}),
	\end{aligned}
\end{equation}
and substituting into \eqref{4.12}, \eqref{4.12} becomes a discrete $\sigma$-form of the Painlev\'{e} IV equation \eqref{4.15}. Substituting \eqref{4.31} into \eqref{4.13}, $\widetilde{\sigma}_{n}$ satisfies a Jimbo-Miwa-Okamoto $\sigma$-form of the Painlev\'{e} IV equation \eqref{4.16}. Letting $t=2s-2\widetilde{T}$,~$\sigma_{n}(s,t)=\dfrac{1}{2}\widehat{\sigma}_{n}+n(s-\widetilde{T})$ and plugging into \eqref{4.14}, then $\widehat{\sigma}_{n}$ satisfies a Jimbo-Miwa-Okamoto $\sigma$-form of the Painlev\'{e} IV equation \eqref{4.17}.
\end{proof}

\section{Asymptotics for the recurrence coefficients and the Hankel determinant}
To derive the asymptotics of recurrence coefficients as $n \rightarrow \infty$, we first briefly described Dyson's Coulomb fluid approach, see e.g. \cite{CI1997,CL1998, CM2002}. The total energy of the repelling charged particles, confined by a common external potential $v_{0}(x)$ reads
\begin{equation}\label{5.1}
\begin{aligned}
E\left(x_1, x_2, \ldots, x_n\right)=-2 \sum_{1 \leq j<k \leq n} \mathrm{ln} \left|x_j-x_k\right|+\sum_{j=1}^n v_{0}\left(x_j\right).
\end{aligned}
\end{equation}
For large enough $n$, the particles can be approximated as a continuous fluid with a certain density $\sigma(x)$, which supported on a single interval $(a, b) $ \cite{Dyson}.  The equilibrium density of the fluid is obtained by the constrained minimization of the free energy function $F[\sigma]$,
\begin{equation}\label{5.2}
\begin{aligned}
F[\sigma]:=\int_a^b \sigma(x) v_{0}(x) d x-\int_a^b \int_a^b \sigma(x) \mathrm{ln} |x-y| \sigma(y) d x d y
\end{aligned}
\end{equation}
subject to
\begin{equation}\label{5.3}
\begin{aligned}
\int_a^b \sigma(x) d x=n, \quad \sigma(x)>0.
\end{aligned}
\end{equation}
Upon minimization, the equilibrium density $\sigma(x)$ is found to satisfy the integral equation
\begin{equation}\label{5.4}
\begin{aligned}
L:=v_{0}(x)-2 \int_a^b \mathrm{ln} |x-y| \sigma(y) d y, \quad x \in[a, b],
\end{aligned}
\end{equation}
where $L$ is the Lagrange multiplier.

 $F[\sigma]$ and $L$ have the following relation \cite{CI1997},
\begin{equation}\label{5.4a}
\begin{aligned}
\frac{\partial F}{\partial n}=L.
\end{aligned}
\end{equation}

To derivative \eqref{5.4} over $x$ both side gives rise to a singular integral equation
\begin{equation}\label{5.5}
\begin{aligned}
v_{0}^{\prime}(x)-2 \text { p.v. } \int_a^b \frac{\sigma(y)}{x-y} d y=0, \quad x \in[a, b],
\end{aligned}
\end{equation}
where p.v. denotes the Cauchy principal value. Because the potential $v_{0}(x)$ is convex and $\sigma(x)$ is supported in a single interval $(a,b)$, the solution of \eqref{5.5} reads
\begin{equation}\label{5.6}
\begin{aligned}
\sigma(x)=\frac{1}{2 \pi^2} \sqrt{\frac{b-x}{x-a}} \mathrm{p} .\mathrm{v} . \int_a^b \frac{v_{0}^{\prime}(y)}{y-x} \sqrt{\frac{y-a}{b-y}} d y .
\end{aligned}
\end{equation}
Thus the normalization
\begin{equation}\label{5.6a}
\begin{aligned}
\int_a^b \sigma(x) d x=n,
\end{aligned}
\end{equation}
becomes
\begin{equation}\label{5.7}
\begin{aligned}
\frac{1}{2 \pi} \int_a^b \sqrt{\frac{y-a}{b-y}} v_{0}^{\prime}(y) d y=n.
\end{aligned}
\end{equation}
Putting $a=s$ and $v_{0}^{\prime}(x)=2x-t$, we find from \eqref{5.7} that
\begin{equation}\label{5.8}
\begin{aligned}
3b^2-2(s+t)b-s^2+2st-8n=0.
\end{aligned}
\end{equation}
Since $b>0$, we have
$$b=\frac{1}{3}\bigg(s+t+\sqrt{(s+t)^2+3(s^2-2st+8n)}\bigg).$$
Hence, we assume the following expansion,
$$b=\frac{2\sqrt{6n}}{3}+\sum\limits_{i=0}^{\infty}\tau_{i}n^{-\frac{i}{2}}.$$
Substituting this into \eqref{5.8} and comparing the corresponding coefficients on both sides, we get
\begin{equation}\label{5.9a}
\begin{aligned}
b=\frac{2\sqrt{6n}}{3}+\frac{t+s}{3}+\frac{\sqrt{6}(2s-t)^2}{72\sqrt{n}}-\frac{\sqrt{6}(2s-t)^4}{6912n^{\frac{3}{2}}}+\frac{\sqrt{6}(2s-t)^6}{331776n^{\frac{5}{2}}}+\mathcal{O}(n^{-\frac{7}{2}}).
\end{aligned}
\end{equation}

It is shown in \cite{CI1997} that
\begin{equation*}
\begin{aligned}
\alpha_{n}\simeq\frac{b+a}{2},\quad \beta_{n}\simeq\bigg(\frac{b-a}{4}\bigg)^2, \quad n\rightarrow\infty.
\end{aligned}
\end{equation*}
In our problem,
\begin{equation*}
\begin{aligned}
\alpha_{n}\simeq\frac{b+s}{2}=\frac{1}{6}(4s+t)+\frac{\sqrt{6n}}{3},\quad \beta_{n}\simeq\bigg(\frac{b-s}{4}\bigg)^2=\frac{n}{6}+\frac{(t-2s)\sqrt{6n}}{36}+\frac{(t-2s)^2}{144}, \quad n\rightarrow\infty.
\end{aligned}
\end{equation*}
For $n\rightarrow\infty$, we assume
\begin{equation}\label{5.9}
\begin{aligned}
\alpha_{n}=\frac{\sqrt{6n}}{3}+\sum\limits_{i=0}^{\infty}a_{i}n^{-\frac{i}{2}},
\end{aligned}
\end{equation}
\begin{equation}\label{5.10}
\begin{aligned}
\beta_{n}=\frac{n}{6}+\sum\limits_{j=-1}^{\infty}b_{j}n^{-\frac{j}{2}}.
\end{aligned}
\end{equation}

\begin{theorem}
The recurrence coefficients $\alpha_{n}$ and $\beta_{n}$ have the following expansions:
\begin{equation}\label{5.13}
\begin{aligned}
\alpha_{n}=\frac{\sqrt{6n}}{3}+\frac{2s}{3}+\frac{t}{6}+\frac{\sqrt{6}(2s-t)^2+12}{144n^{\frac{1}{2}}}+\mathcal{O}(n^{-\frac{3}{2}}),
\end{aligned}
\end{equation}
\begin{equation}\label{5.14}
\begin{aligned}
\beta_{n}=\frac{n}{6}-\frac{\sqrt{6}(2s-t)\sqrt{n}}{36}+\frac{(2s-t)^2}{72}-\frac{\sqrt{6}(2s-t)^3}{1728n^{\frac{1}{2}}}+\mathcal{O}(n^{-1}).
\end{aligned}
\end{equation}
\end{theorem}

\begin{proof}
 Substituting \eqref{2.5a} and \eqref{2.6c} into \eqref{2.5b}, we have
\begin{equation}\label{5.15}
\begin{aligned}
2\beta_{n+1}+2\beta_{n}-2n-1=(s-\alpha_{n})(2\alpha_{n}-t).
\end{aligned}
\end{equation}
Inserting \eqref{2.44a} into \eqref{2.44b} produces
\begin{equation}\label{5.16}
\begin{aligned}
(\alpha_{n}-s)(2\beta_{n+1}-2\beta_{n}-1)=\beta_{n}R_{n-1}-\beta_{n+1}R_{n+1}.
\end{aligned}
\end{equation}
In view of \eqref{2.5a}, it follows that
\begin{equation}\label{5.17}
\begin{aligned}
R_{n-1}=2\alpha_{n-1}-t,
\end{aligned}
\end{equation}
\begin{equation}\label{5.18}
\begin{aligned}
R_{n+1}=2\alpha_{n+1}-t.
\end{aligned}
\end{equation}
Substituting \eqref{5.17} and \eqref{5.18} into \eqref{5.16} gives
\begin{equation}\label{5.19}
\begin{aligned}
(\alpha_{n}-s)(2\beta_{n+1}-2\beta_{n}-1)=\beta_{n}(2\alpha_{n-1}-t)-\beta_{n+1}(2\alpha_{n+1}-t).
\end{aligned}
\end{equation}
Replacing $n$ by $n\pm1$ in \eqref{5.9} and \eqref{5.10}, respectively, we have
\begin{equation}\label{5.18a}
\begin{aligned}
\alpha_{n\pm1}&=\frac{\sqrt{6(n\pm1)}}{3}+\sum\limits_{i=0}^{\infty}a_{i}(n\pm1)^{-\frac{i}{2}}\\
&=\frac{\sqrt{6n}}{3}(1\pm\frac{1}{n})^{\frac{1}{2}}+a_{0}+\sum\limits_{i=1}^{\infty}a_{i}n^{-\frac{i}{2}}(1\pm \frac{1}{n})^{-\frac{i}{2}}\\
&=\frac{\sqrt{6n}}{3}+a_{0}+\dfrac{\pm\dfrac{\sqrt{6}}{6}+a_{1}}{n^{\frac{1}{2}}}+\dfrac{a_{2}}{n}+\dfrac{-\dfrac{\sqrt{6}}{24}\mp\dfrac{a_{1}}{2}+a_{3}}{n^{\frac{3}{2}}}+\mathcal{O}(n^{-2}),
\end{aligned}
\end{equation}
and
\begin{equation}\label{5.19a}
\begin{aligned}
\beta_{n+1}&=\frac{n+1}{6}+\sum\limits_{i=-1}^{\infty}b_{i}(n\pm1)^{-\frac{i}{2}}\\
&=\frac{n+1}{6}+b_{-1}\sqrt{n}(1+\frac{1}{n})^{\frac{1}{2}}+b_{0}+\sum\limits_{i=1}^{\infty}b_{i}n^{-\frac{i}{2}}(1+ \frac{1}{n})^{-\frac{i}{2}}\\
&=\frac{n+1}{6}+b_{-1}\sqrt{n}+b_{0}+\dfrac{\dfrac{b_{-1}}{2}+b_{1}}{n^{\frac{1}{2}}}+\dfrac{b_{2}}{n}+\dfrac{-\dfrac{b_{-1}}{8}-\dfrac{b_{1}}{2}+b_{3}}{n^{\frac{3}{2}}}+\mathcal{O}(n^{-2}).
\end{aligned}
\end{equation}
Substituting \eqref{5.18a} and \eqref{5.19a} into \eqref{5.15} and \eqref{5.19}, we have
\begin{equation*}
\begin{aligned}
\sum\limits_{i=-1}^{\infty}\frac{d_{i}}{n^{\frac{i}{2}}}=0,\quad \sum\limits_{i=-1}^{\infty}\frac{e_{i}}{n^{\frac{i}{2}}}=0,
\end{aligned}
\end{equation*}
where $d_{i}$ and $e_{i}$ depends on the coefficients $a_{i}$, $b_{i}$, $t$ and $s$. In order to satisfy the above equations, all the coefficients of powers of $n$ are identically zero. After simplify, $e_{-1}=0$ is identical.

The equations $d_{-1}=0$ and $e_{0}=0$ give rise to
\begin{equation*}
\begin{aligned}
4b_{-1}-\frac{\sqrt{6}(2s-t)}{3}+\frac{4\sqrt{6}}{3}a_{0}=0,
\end{aligned}
\end{equation*}
and
\begin{equation*}
\begin{aligned}
-\frac{4\sqrt{6}}{3}b_{-1}-\frac{2s-t}{3}+\frac{1}{3}a_{0}=0.
\end{aligned}
\end{equation*}
Solving the above two equations, we find
\begin{equation*}
\begin{aligned}
a_{0}=\frac{2s}{3}+\frac{t}{6}, \quad  b_{-1}=-\frac{\sqrt{6}(2s-t)}{36}.
\end{aligned}
\end{equation*}
With the values of $a_{0}$ and $b_{-1}$, the equations $d_{0}=0$ and $e_{1}=0$ show
\begin{equation*}
\begin{aligned}
4b_{0}+\frac{4\sqrt{6}}{3}a_{1}-\frac{2}{3}-\frac{(2s-t)^2}{9}=0,
\end{aligned}
\end{equation*}
and
\begin{equation*}
\begin{aligned}
-\frac{2\sqrt{6}}{3}b_{0}-\frac{\sqrt{6}}{18}+\frac{\sqrt{6}(2s-t)^2}{216}+\frac{2}{3}a_{1}=0.
\end{aligned}
\end{equation*}
Solving the above two equations, we find
\begin{equation*}
\begin{aligned}
a_{1}=\frac{\sqrt{6}(2s-t)^2}{144}+\frac{\sqrt{6}}{12}, \quad  b_{0}=\frac{(2s-t)^2}{72}.
\end{aligned}
\end{equation*}
This procedure can be easily extended to find higher coefficients $a_{2}, a_{3},\cdots,$ and $b_{1}, b_{2}, b_{3},\cdots.$ We only list some of them
\begin{equation*}
\begin{aligned}
a_{2}=0, \quad  b_{1}=-\frac{\sqrt{6}(2s-t)^3}{1728}.
\end{aligned}
\end{equation*}

\end{proof}

\begin{lemma}
We have
\begin{equation}\label{5.20}
\begin{aligned}
L=\frac{1}{8}[3b^2+2bs+3s^2-4t(b+s)]-2n\mathrm{ln}\frac{b-s}{4}.
\end{aligned}
\end{equation}
When $n\rightarrow\infty$, we obtain
\begin{equation}\label{5.211}
\begin{aligned}
L=&-n\mathrm{ln}n+(1+\mathrm{ln}6)n+\frac{\sqrt{6}}{3}(2s-t)\sqrt{n}+\frac{1}{12}(8s^2-8st-t^2)+\frac{\sqrt{6}(2s-t)^3}{432n^{\frac{1}{2}}}\\
&-\frac{(2s-t)^4}{1152n}+\frac{\sqrt{6}(2s-t)^5}{13824n^{\frac{3}{2}}}+\frac{(2s-t)^6}{165888n^{2}}+\mathcal{O}(n^{-\frac{5}{2}}).
\end{aligned}
\end{equation}
and
\begin{equation}\label{5.222}
\begin{aligned}
F[\sigma]=&-\frac{1}{2}n^2\mathrm{ln}n+\frac{(3+2\mathrm{ln}6)}{4}n^2+\frac{2\sqrt{6}}{9}(2s-t)n^{\frac{3}{2}}+\frac{1}{12}(8s^2-8st-t^2)n+\frac{\sqrt{6}(2s-t)^3n^{\frac{1}{2}}}{216}\\
&-\frac{(2s-t)^4\mathrm{ln}n}{1152}+C_{0}-\frac{\sqrt{6}(2s-t)^5}{6912n^{\frac{1}{2}}}-\frac{(2s-t)^6}{165888n}+\mathcal{O}(n^{-\frac{3}{2}}),
\end{aligned}
\end{equation}
where $C_{0}$ is a constant independent of $n$.

\end{lemma}

\begin{proof}
Putting $a=s$, multiplying both sides by $\frac{1}{\sqrt{(b-x)(x-s)}}$ and integrating from $s$ to $b$ gives rise to
\begin{equation}\label{5.21}
\begin{aligned}
\pi L=\int_{s}^{b}\frac{v_{0}(x)}{\sqrt{(b-x)(x-s)}}dx-2 \int_a^b \sigma(y) d y\int_{s}^{b}\frac{\mathrm{ln} |x-y|}{\sqrt{(b-x)(x-s)}}dx , \quad x \in[s, b],
\end{aligned}
\end{equation}
where we used the integral identities
\begin{equation*}
\begin{aligned}
\int_{s}^{b}\frac{1}{\sqrt{(b-x)(x-s)}}dx=\pi.
\end{aligned}
\end{equation*}
With $v_{0}(x)=x^2-tx$, after some computations, we have
\begin{equation}\label{5.22}
\begin{aligned}
\int_{s}^{b}\frac{v_{0}(x)}{\sqrt{(b-x)(x-s)}}dx=\frac{1}{8}[3b^2+2bs+3s^2-4t(b+s)].
\end{aligned}
\end{equation}

Denote that
\begin{equation*}
\begin{aligned}
\int_{s}^{b}\frac{\mathrm{ln} |x-y|}{\sqrt{(b-x)(x-s)}}dx=C,
\end{aligned}
\end{equation*}
where $C$ is independent of $y$. Then we replace $y$ by $b$ to compute $C$
\begin{equation}\label{5.23}
\begin{aligned}
\int_{s}^{b}\frac{\mathrm{ln} (b-x)}{\sqrt{(b-x)(x-s)}}dx=\int_{0}^{1}\frac{\mathrm{ln} [(b-s)(1-u)]}{\sqrt{u(1-u)}}du=\pi\mathrm{ln}\frac{b-s}{4} .
\end{aligned}
\end{equation}
where we used the integral identities
\begin{equation*}
\begin{aligned}
\int_{1}^{0}\frac{\mathrm{ln} (1-u)}{\sqrt{u(1-u)}}du=-2\pi\mathrm{ln}2.
\end{aligned}
\end{equation*}
Substituting \eqref{5.22} and \eqref{5.23} into \eqref{5.21}, we obtain the desired result, where we used the fact \eqref{5.6a}.

Inserting \eqref{5.9} into \eqref{5.20} and letting $n\rightarrow\infty$, we find \eqref{5.211}.

Bearing in mind of \eqref{5.4a}, we obtain the \eqref{5.222}.

\end{proof}

\section{Large $n$ behavior of orthogonal polynomials}
By making use of the behaviors about recurrence coefficients, we consider the large $n$ behavior of the monic orthogonal polynomials $P_{n}$ and show that$P_{n}$ satisfies the biconfluent Heun equation by using the methods of previous studies \cite{ZC2019}, \cite{YLZC} and \cite{YCLZC}.
\begin{theorem}
As $n\rightarrow\infty$, then $u(z)=P_{n}(\frac{z}{\sqrt{2}}+s)$ satisfies the biconfluent Heun equation \cite{Ronveaux}
\begin{equation}\label{6.1}
\begin{aligned}
u^{\prime \prime}(z)-\left(z+\frac{\gamma}{z}+\delta\right) u^{\prime}(z)+\frac{\alpha z-q}{z}u(z)=0,
\end{aligned}
\end{equation}
where $\gamma=-1, \delta=\frac{\sqrt{2}}{2}(2s- t), \alpha=0, q=-\frac{4 \sqrt{3} n^{\frac{3}{2}}}{9}$.
\end{theorem}
\begin{proof}
Inserting \eqref{2.6a} and \eqref{2.6c} into \eqref{2.6b} to eliminate $R_{n-1}$ and $\beta_{n}$ yields
\begin{equation}\label{6.2}
\begin{aligned}
\sum_{j=0}^{n-1}R_j=(r_{n}+n)R_n+\frac{2r^2_n}{R_n}-(2s-t)r_n.
\end{aligned}
\end{equation}
Substituting \eqref{2.2a}, \eqref{2.2b} into \eqref{2.555}, and using \eqref{6.2} to eliminate $\sum_{j=0}^{n-1}R_j$, we have
\begin{equation}\label{6.3}
\begin{aligned}
&P''_{n}(x)-\bigg(2 x-t-\frac{R_n}{(x-s)(2 x-2s+R_n)}\bigg)P_{n}'(x)+\bigg[-\frac{r_{n}}{(z-s)^2}\\
&~~~~~~~~~~~~+\frac{r_{n}R_{n}}{(x-s)^2(2 x-2s+R_n)}+2 n+\frac{(r_{n}+n)R^2_n+2r^2_n-(2s-t)r_nR_n}{(x-s)R_n}\bigg]P_{n}(x)=0.
\end{aligned}
\end{equation}
Based on the expansions $\alpha_{n}$ and $\beta_{n}$ in \eqref{5.9} into \eqref{5.10}, we can obtain the expansions $R_{n}$ and $r_{n}$ by using \eqref{2.5b} and \eqref{2.6c}
\begin{equation}\label{6.4}
\begin{aligned}
&R_n(x ; s, t)=\frac{2\sqrt{6n}}{3}+\frac{2s}{3}-\frac{5t}{6}+\mathcal{O}(n^{-\frac{1}{2}}), \\
&r_n(x ; s, t)=-\frac{2n}{3}+\frac{2\sqrt{6}(2s-t)\sqrt{n}}{18}+\frac{2s}{3}+\frac{(2s-t)^2}{36}+\mathcal{O}(n^{-\frac{1}{2}}).
\end{aligned}
\end{equation}
Plugging \eqref{6.4} into \eqref{6.3}, we have
\begin{equation}\label{6.3}
\begin{aligned}
&P''_{n}(x)-\bigg(2 x-t-\frac{1}{x-s}\bigg)P_{n}'(x)+\frac{4\sqrt{6}n^{\frac{3}{2}}}{9(x-s)}P_{n}(x)=0.
\end{aligned}
\end{equation}
Letting
\begin{equation*}
\begin{aligned}
x=\frac{z}{\sqrt{2}}+s,
\end{aligned}
\end{equation*}
then $u(z):=P_{n}(\frac{z}{\sqrt{2}}+s)$ satisfies the biconfluent Heun equation \eqref{6.1} with parameters  $\gamma=-1, \delta=\frac{\sqrt{2}}{2}(2s- t), \alpha=0, q=-\frac{4 \sqrt{3} n^{\frac{3}{2}}}{9}$.
\end{proof}

\begin{remark}
The Heun equation and its four standard confluent forms is very important in mathematical physics. The four standard confluent Heun equations are confluent Heun equation
, doubly confluent Heun equation, biconfluent Heun equation and triconfluent Heun equation \cite{Ronveaux}. The solutions of these Heun equations are special functions, such as  hypergeometric functions, Mathieu functions and spheroidal functions.
\end{remark}

\section{Acknowledgements}

M. Zhu acknowledges the support of the National Natural Science Foundation of China (Grant No. 12201333) and the Breeding Plan of Shandong Provincial Qingchuang Research Team (Grant No. 2023KJ135).

\section{Data availability}
The data that support the findings of this study are available from the corresponding author upon reasonable request.

\end{document}